\documentclass{article}

\usepackage[margin=3.2cm]{geometry}

\usepackage{setspace}
\usepackage{graphics}

\usepackage{docmute}
\usepackage[utf8]{inputenc}

\usepackage{mathtools}
\usepackage{amsmath}
\usepackage{amsfonts}
\usepackage{amssymb}
\usepackage{amsthm}
\usepackage{aliascnt}

\usepackage{makeidx}
\makeindex

\usepackage{textcomp}
\usepackage[sb]{libertine}
\usepackage[varqu,varl]{zi4}%
\usepackage[libertine,bigdelims,vvarbb]{newtxmath}
\usepackage[supsfam=LibertinusSerif-Sup,supscaled=1.2,raised=-.13em]{superiors}
\usepackage{bm}
\useosf

\usepackage[breaklinks,linktocpage]{hyperref}
\usepackage[hyperpageref]{backref}

\usepackage{url}
\usepackage{breakurl}

\usepackage{tikz-cd}

\usepackage{enumitem}
\usepackage[UKenglish]{babel}
\usepackage[makeroom]{cancel}
\usepackage{pstricks-add}

\definecolor{shadecolor}{rgb}{0.88,0.91,0.95}
\usepackage{tcolorbox}

\usepackage{dynkin-diagrams}
\usetikzlibrary{backgrounds}

\usepackage{booktabs}

\usepackage{fancyhdr}
\newcommand{\Z}{\mathbb{Z}}

\newcommand{\R}{\mathbb{R}}

\newcommand{\spann}{\operatorname{span}}

\newcommand{\Stab}{\operatorname{Stab}}

\newcommand{\Ric}{\operatorname{Ric}}

\newcommand{\Id}{\operatorname{Id}}

\renewcommand{\|}[1]{\left| \left| #1 \right| \right|}

\newcommand{\<}{\langle}
\renewcommand{\>}{\rangle}

\newcommand{\vol}{\operatorname{vol}}

\newcommand\Item[1][]{% 
  \ifx\relax#1\relax  \item \else \item[#1] \fi
  \abovedisplayskip=0pt\abovedisplayshortskip=0pt~\vspace*{-\baselineskip}}

\usepackage[capitalise]{cleveref}

\numberwithin{equation}{section}

\newtheorem{proposition}[equation]{Proposition}
\crefname{proposition}{Proposition}{Propositions}

\crefname{lemma}{Lemma}{Lemmas}

\newtheorem{corollary}[equation]{Corollary}
\crefname{corollary}{Corollary}{Corollaries}

\newtheorem*{corollary*}{Corollary}
\crefname{corollary*}{Corollary}{Corollaries}

\newtheorem{theorem}[equation]{Theorem}
\crefname{theorem}{Theorem}{Theorems}

\crefname{exercise}{Exercise}{Exercises}

\crefname{task}{Task}{Task}

\crefname{conjecture}{Conjecture}{Conjectures}

\crefname{algorithm}{Algorithm}{algorithms}

\newtheorem*{theorem*}{Theorem}
\crefname{theorem*}{Theorem}{Theorems}

\crefname{claim}{Claim}{Claims}

\theoremstyle{remark}

\crefname{question}{Question}{Questions}

\crefname{definition}{Definition}{Definitions}

\crefname{example}{Example}{Examples}

\crefname{remark}{Remark}{Remarks}

\crefname{assumption}{Assumption}{Assumptions}

\crefname{problem}{Problem}{Problems}
\usepackage{subcaption}

\author{Daniel Platt}
\date{\today}
\title{Non-uniqueness and symmetries for the Nirenberg problem using computer assistance}

\makeatletter
\@ifundefined{code}{%
  \newcommand{\code}[1]{\texttt{\detokenize{#1}}}%
}

\usepackage{booktabs}
\usepackage{siunitx}
\usepackage{float}
\begin{document}

\maketitle

\begin{abstract}
    We apply verified numerics to the Nirenberg problem, proving that a genuine solution exists near two given computer-generated approximate solutions.
    This proves existence of a solution for a particular prescribed curvature that was previously predicted, but not proved, to exist.
    We are also able to determine the symmetry groups of the genuine solutions exactly, which in one case is different from the symmetry of the prescribed curvature.
    We expect the computer code for this proof can be reused to study other aspects of the Nirenberg problem.
\end{abstract}

\tableofcontents

\section{Introduction}

The classical \emph{Nirenberg problem} asks:
which functions $K$ on $S^2$ are realised as the curvature of a metric $g$ on $S^2$ that is pointwise conformal to the round metric $g_{\text{round}}$ of radius $1$?
Writing $g=e^{2u} g_{\text{round}}$, the curvature condition is equivalent to $u$ solving the equation
\begin{align}
\label{equation:nirenberg-pde}
    \mathscr{F}(u)
    :=
    1-\Delta u-K e^{2u}
    =0.
\end{align}
The problem has received much attention over the years.
An early landmark paper is \cite{Koutroufiotis1972}, in which PDE techniques were applied to the problem that would become standard in later treatments, and a sufficient condition for existence was proved.
Other far-reaching sufficient conditions in the literature are \cite{Moser1973,Chen1987}.

Abstractly, the moduli space of solutions is quite well understood. 
By the one–point concentration theory in dimension $2$ from \cite{Chang1987,Chang1988}, one may compactify the space of solutions in a natural sense.

In practice, it is hard to decide for a given function $K$ whether a solution exists, and how many.
For $K \equiv 1$, an infinite family of solutions exhibited in \cite{Onofri1982} exists.
In \cite{Anderson2021}, bifurcation of solutions was studied, which can lead to situations with non-unique solutions.
In particular, many $K$ near $K \equiv 1$ admit an even number of solutions, i.e. at least two.
Proving \emph{non-uniqueness} of solutions is one of the central problems in the area, see e.g. Struwe's problem statement in \cite[pp.57-58]{Rassias2022}.

A different problem is studying the \emph{symmetries} of solutions:
pioneered in \cite{Moser1973}, it has become a common theme to prove existence of a solution under some symmetry assumption on $K$, see
\cite{Hong1986,Chang1988}.
In these constructions, the constructed solutions typically inherit the same symmetries as $K$.
An exception is \cite{Xu1993}, where symmetry of $K$ is used as an input, but the symmetry of the solution is not known in general.
That raises the question:
given a function $K$ which has a symmetry, do all solutions to \cref{equation:nirenberg-pde} have that symmetry?
The case $K \equiv 1$ shows that the answer to this question is \emph{No} in general, however this may be the only documented example of such symmetry breaking.

Approaching the problem from a different angle, different numerical methods have been applied to the problem.
In \cite{Sun2015,Campen2021}, numerics for a discrete version of the Nirenberg problem were implemented.
In \cite{Liu2025} (building on the works \cite{Luo2004,Bobenko2015}), discretisation was used to approximately solve the smooth version of the problem.
In \cite{Cortes2026}, machine learning was used to approximately solve the smooth version.
Therein, \emph{existence and non-existence} for several explicit $K$ is conjectured based on numerical results.

In this article, we make progress on these three problems of \emph{non-uniqueness, symmetry}, and \emph{existence} of solutions.
We consider $K=Y_{3,2}$, i.e. a real spherical harmonic of bi-degree $(3,2)$, shown in \cref{figure:spheres}.
It is symmetric under the tetrahedral group $T_d \subset O(3)$ of order $24$.
We show that:
five different solutions to \cref{equation:nirenberg-pde} exist, thereby proving a conjecture from \cite[Figure 7]{Cortes2026}, and study the symmetry groups of the solutions.
This is summarised in our main theorem:

\begin{theorem}
    \label{theorem:main-existence}
    Let $K=Y_{3,2}$ be a unit-norm spherical harmonic of bi-degree $(3,2)$.
    Then there exist at least five distinct solutions to \cref{equation:nirenberg-pde}.
    One solution has a $T_d$-symmetry and the other $4$ solutions have a $S_3$-symmetry only.
\end{theorem}

\begin{figure*}[t]
\centering
\includegraphics[width=0.9\textwidth]{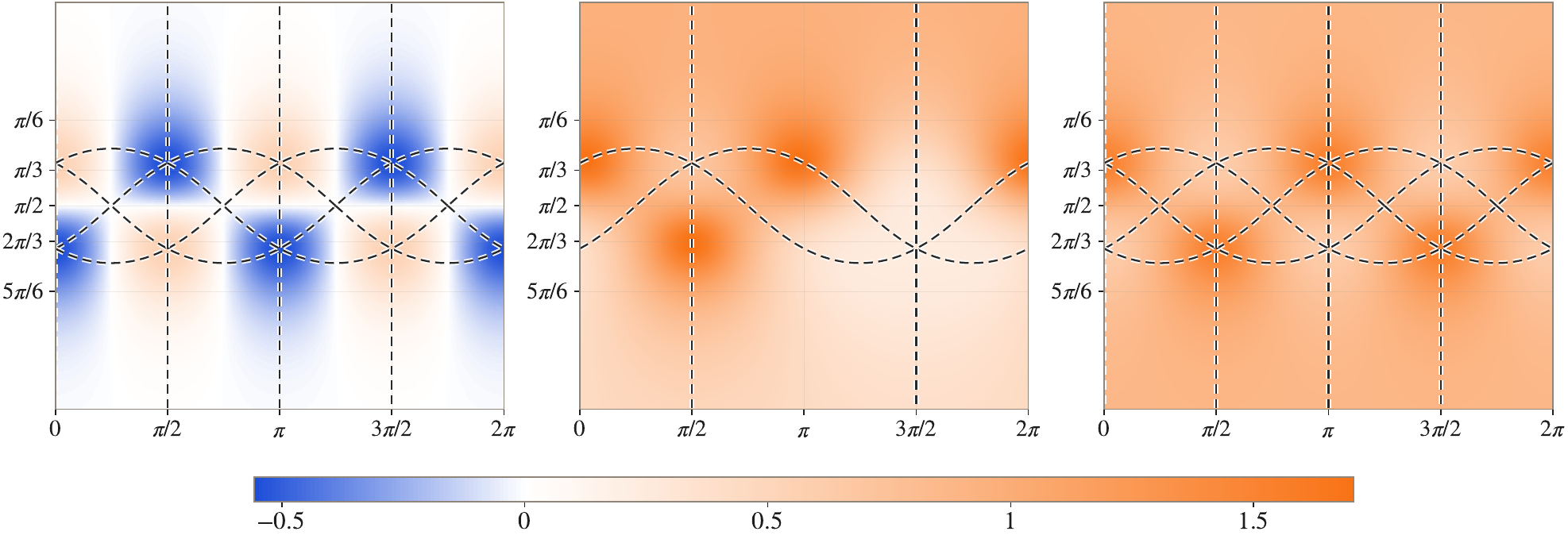}
\caption{Left to right: $K=Y_{3,2}$, an approximate solution to \cref{equation:nirenberg-pde} with $S_3$-symmetry, an approximate solution to \cref{equation:nirenberg-pde} with $T_d$-symmetry in Mercator projection. Dashed lines show the symmetry planes of the three functions. For $Y_{3,2}$ these are genuine symmetries. The approximate solutions are \emph{not} exactly invariant under these ambient symmetries. However, it is shown in the proof of \cref{theorem:main-existence} that genuine solutions to \cref{equation:nirenberg-pde} exist near the approximate solutions which have precisely the shown symmetry groups. The approximate solutions have the notation $u_0'$ in \cref{subsection:finding-an-approximate}.}
\label{figure:spheres}
\end{figure*}

The theorem is proved through a numerically verified proof.
First, approximate solutions $u_0$ (one with $D_4$-symmetry and one with $\mathbb{Z}_2$-symmetry) are computed.
Second, a bound for the residual $\mathscr{F}(u_0)$ and for the inverse of the linearised operator $\mathscr{L}^{-1}$ is proved.
This is done by combining an a priori estimate from standard elliptic analysis with a bound for the spectral gap.
If the residual is small compared to this operator norm, then it follows from a routine application of a fixed point theorem that a genuine solution to \cref{equation:nirenberg-pde} exists near $u_0$.

This scheme of approximating, estimating the linearised operator, and then perturbing to a genuine solution was pioneered for applications in geometry in \cite{Taubes1982} and has been applied to numerous problems since then.
These are usually pen-and-paper proofs, where the approximate solutions $u_0$ can be written down explicitly.
More recently, it has become possible to use computer-generated approximate solutions $u_0$ together with verified numerics to execute this scheme.
Applications in geometry are
\cite{Balazs2018,Wang2025,Reiterer2019} and \cite{Buttsworth2026}, which was an important influence for this work.

We then go one step further and determine the symmetry groups of our solutions.
Showing that a certain symmetry is present is standard in numerically verified proofs.
It is usually achieved by reducing the analysis to a subspace of symmetric functions, as done in all of the four numerically verified proofs in geometry listed above.
There is also some work on ruling out certain symmetries of solutions, see e.g. \cite{Arioli2015,Arioli2019}.
Though it was not pursued, precisely determining the stabiliser of a solution would have been possible in \cite{Salort2022}, but other than we are not aware of numerically verified proofs that achieve this.

Our approach is general and works for other choices of $K$ and other symmetry groups.
The provided computer code can be easily adapted to certify other numerical calculations pertaining to the Nirenberg problem.

The paper is structured as follows:
in \cref{section:background} we fix some notation for the Nirenberg problem and spherical harmonics.
In \cref{section:construction} we give the proof of \cref{theorem:main-existence}.
The explicit numerical values verified by our computer code have all been collected in \cref{section:appendix-1}.
Thus, many statements in the main body of the article are of the form \emph{``the quantity [\dots] is bounded by the value given in \cref{section:appendix-1}"}.
The Python code used to run the verified numerics can be found at \url{https://github.com/danielplatt/nirenberg-certified}.

\textbf{Acknowledgments.}
The author thanks Timothy Buttsworth for suggesting the proof of the a priori estimate \cref{equation:a-priori}. 
This research has been supported by EPSRC grant EP/Y028872/1, Mathematical Foundations of Intelligence: An Erlangen Programme for AI.

\section{Background}
\label{section:background}

\subsection{The Nirenberg Problem}

We decompose the partial differential equation from \cref{equation:nirenberg-pde} into linear and non-linear parts as follows.
For $u_0 \in C^\infty(S^2)$ we write

\begin{align}
\label{equation:linearisation-higher-order-def}
    \begin{split}
        \mathscr{F}(u_0+u)
        &=
        \mathscr{F}(u_0)
        +
        \mathscr{L}(u)
        +
        \mathscr{N}(u),
        \quad
        \text{where}
        \\
        \mathscr{L}(u)
        &=
        \Delta u-2K e^{2u_0} \cdot u,
        \\
        \mathscr{N}(u)
        &=
        -K e^{2u_0} (e^{2u}-1-2u).
    \end{split}
\end{align}

\subsection{Spherical Harmonics}
\label{subsection:spherical-harmonics}

For $l \geq 0$ there exist $2l+1$ spherical harmonics of degree $l$.
We choose the usual orthonormal basis 
\begin{align}
\label{theorem:spherical-harmonics}
Y_{\ell m}^{\mathbb{R}}(\theta,\phi)=
\begin{cases}
\sqrt{2}\,N_{\ell m}\,P_\ell^{m}(\cos\theta)\cos(m\phi), & m=1,\dots,\ell,\\[2mm]
N_{\ell 0}\,P_\ell(\cos\theta), & m=0,\\[2mm]
\sqrt{2}\,N_{\ell |m|}\,P_\ell^{|m|}(\cos\theta)\sin(|m|\phi), & m=-1,\dots,-\ell,
\end{cases}
\end{align}
\[
N_{\ell m}
=
\sqrt{\frac{2\ell+1}{4\pi}\frac{(\ell-m)!}{(\ell+m)!}},
\qquad
\ell=0,1,2,\dots
\]
from \cite[Equation 4.7]{Atkinson2012}.
In this formula, $P_l$ are Legendre polynomials and $P_l^{|m|}$ are associated Legendre functions.
The functions $Y_{lm}$ are eigenfunctions for the Laplace operator with eigenvalue $\lambda(l):=l(l+1)$.

The symmetry group of spherical harmonics was worked out in \cite{Altmann1957}.
Let
\[
R=\begin{pmatrix}
    \cos (\pi/4) & -\sin (\pi/4) &0 \\
    \sin(\pi/4) & \cos(\pi/4) &0 \\
    0&0&1
\end{pmatrix}
\]
be a rotation of $45^\circ$ degrees in the $(x,y)$-plane and write $(x,y,z)=R(u,v,w)$.
For the curvature prescriber $K=Y_{3,2}$ from \cref{theorem:main-existence} we get:

\begin{proposition}
\label{proposition:Y32-symmetry}
    For the symmetry group of $Y_{3,2}$ we have
    \[
    \Stab_{O(3)} Y_{3,2}
    =
    R T_d R^{-1},
    \]where
    \begin{align*}
        T_d
        &=
        \{
        \text{signed permutations of $(x,y,z)$ with an even number of sign changes}
        \}.
    \end{align*}
\end{proposition}

\begin{proof}
    A quick calculation shows that $Y_{32}=c \cdot z(x^2-y^2)$ in Cartesian coordinates for a constant $c \in \R$, see \cite[Table 7.1]{Ceulemans2024}.
    Thus, in coordinates in which $x$ and $y$ are rotated $45^\circ$, i.e. $u=\frac{x+y}{\sqrt{2}}$, $v=\sqrt{x-y}{\sqrt{2}}$, $w=z$, we have
    $Y_{32}=c \cdot uvw$, which has symmetry group $T_d$ (\cite[Table 15.3]{Altmann1957} states that $T_d$ is contained in the symmetry group, but the converse also holds).
    The map $R$ is a rotation of $x$ and $y$ by $45^\circ$, which proves the claim.
\end{proof}

\subsection{Geometric Analysis}

For $u \in C^\infty(S^2)$ and $s \in \Z_{\geq 0}$ we define its Sobolev norm to be
\[
\|{u}_{H^s}
:=
\sum_{j=0}^s
\|{\nabla^j s}_{L^2}.
\]
For this convention, we record here an explicit Sobolev embedding theorem:

\begin{proposition}
\label{proposition:sobolev-embedding}
    If $u \in H^2(S^2)$, then $u \in L^\infty(S^2)$ and
    \begin{align}
        \|{u}_{L^\infty}
        \leq
        C_{\text{emb}}
        \|{u}_{H^2},
    \end{align}
    where $C_{\text{emb}}=\sqrt{\frac{625}{592}} C_{2,0}$ and $C_{2,0}$ is defined in \cref{equation:C20-def}.
\end{proposition}

\begin{proof}
    For $u=\sum_{l=0}^\infty \sum_{m=-l}^l a_{lm} Y_{lm}$ written as a sum of spherical harmonics, the spectral Sobolev norm is defined as
    \[
    \|{u}_{H^2,\text{spectral}}^2
    =
    \sum_{l,m} \left( l+\frac{1}{2} \right)^4 |a_{lm}|^2,
    \]
    see e.g. \cite[p.4]{Hesse2005}.
    By \cite[Equation 2.3]{Hesse2005}:
    \begin{align}
        \|{u}_{L^\infty}
        \leq
        C_{2,0}
        \|{u}_{H^2,\text{spectral}},
    \end{align}
    where 
    \begin{align}
        \label{equation:C20-def}
        C_{2,0}=\left( \frac{1}{2\pi} \sum_{l=0}^\infty \left(
        l+\frac{1}{2}
        \right)^{-3}
        \right).
    \end{align}
    The different definitions of Sobolev norms are well known to be equivalent, and we will derive one explicit norm equivalence constant here.
    One checks using a computer algebra system that
    \[
    \left(
    l+\frac{1}{2}
    \right)^4
    \leq
    \frac{625}{592}
    (1+l^2(l+1)^2)
    \quad
    \text{ for all }
    l \geq 0.
    \]
    Integrating the Bochner formula \cite[Equation 4.5.3]{Jost2005} and using $\Ric=1$ on $S^2$ gives
    \[
    \|{\nabla^2 u}_{L^2}^2+\|{\nabla u}_{L^2}^2
    =
    \|{\Delta u}_{L^2}^2,
    \]
    i.e.
    \begin{align*}
        \|{u}^2_{H^2}
        &=
        \|{u}^2_{L^2}
        +
        \|{\Delta u}^2_{L^2}
        =
        \sum_{l,m}
        (1+\lambda_l^2) |a_{lm}|^2
        =
        \sum_{l,m}
        (1+l^2(l+1)^2) |a_{lm}|^2.
    \end{align*}
    Hence,
    \begin{align*}
        \|{u}_{L^\infty}^2
        &\leq
        C_{2,0}^2
        \|{u}_{H^2,\text{spectral}}^2
        \leq
        C_{2,0}^2
        \sum_{l,m} \left( l+\frac{1}{2} \right)^4 |a_{lm}|^2
        \leq
        C_{2,0}^2
        \frac{625}{592}
        \sum_{l,m}
        (1+l^2(l+1)^2) |a_{lm}|^2
        =
        C_{2,0}^2
        \frac{625}{592}
        \|{u}_{H^2}^2,
    \end{align*}
    which proves the claim.
\end{proof}

An upper bound for the constant $C_{\text{emb}}$ is computed by the function \code{_2_3_C_emb}.

\subsection{Rigorous computation}

\paragraph{Interval arithmetic.}
Typical computer arithmetic using floating point numbers is subject to rounding errors and is not suitable for numerically verified proofs without additional processing.
It is standard in verified computing to use interval arithmetic for this.
That is:
numbers are treated as uncertain, represented by an interval guaranteed to contain the exact unknown value, and all operations (such as addition and multiplication) keep track of this uncertainty, producing a new interval.
\cite[Section 2.5]{Moore2009} surveys the rich history of this idea.
In practice, we use the \texttt{flint} Python-wrapper from \cite{Johansson2025} for the \texttt{C}-package \texttt{arb} \cite{Johansson2017}.

\paragraph{Quadrature.}
We often need to compute integrals of spherical harmonics.
For example, when expressing the product of two spherical harmonics as a linear combination of spherical harmonics, one computes the linear coefficients as integrals.
Quadrature with guaranteed error bounds works in great generality.
However, using Gauss-Legendre quadrature we make use of the special structure of spherical polynomials that does not require this general theory.
For spherical polynomials of degree $2k-1$, this quadrature gives the exact integral if one chooses at least $k \times 2k$ quadrature points, see \cite[Theorem 5.4]{Atkinson2012}.
This has previously been employed in \cite[p.2581]{Wieczorek2018}.
This applies in particular to spherical harmonics.

\section{Construction of the solution}
\label{section:construction}

We will prove \cref{theorem:main-existence} by constructing an approximate solution and using the Banach fixed point theorem to show that a nearby genuine solution exists.
The proof is structured in five steps, each making up one of the subsections in this section.

Throughout, we will use the notation $\overline{x}$ for a rigorous, non-sharp upper bound of some quantity $x$ that will be computed numerically.

\Cref{theorem:main-existence} claims the existence of (up to symmetry) two different solutions to \cref{equation:nirenberg-pde}.
Beginning from two different approximate solutions, the existence proof is completely analogous for both.
In the section text we will only refer to the approximate solutions as if there was only one, writing $u_0$ as a placeholder for either of the two approximate solutions.
The two candidates satisfy different rigorous estimates, and those are distinguished in \cref{table:numerical-values}.

\subsection{Finding an approximate solution $u_0$}
\label{subsection:finding-an-approximate}

In this section, we briefly discuss the approximate solutions $u_0$ to \cref{equation:nirenberg-pde}.
We note that for the rest of the proof it is not important how $u_0$ was obtained and there are many viable numerical schemes to achieve this.
We require no theoretical guarantees for our solver:
the usual questions of convergence and accuracy play no role, because all our analysis is \emph{a posteriori} on $u_0$ only.
We do not supply the computer code of our numerical solver, but only the two different approximate solutions $u_0$ in the hope this may simplify the review process.

It is convenient if $u_0$ is represented as a sum of spherical harmonics, and because of this we implement (a basic version of) the pseudospectral method from \cite{Pfeiffer2003}.
We obtain two candidate solutions $u_0'$ which are finite sums of spherical harmonics up to some chosen degree $N$, i.e.
\begin{align}
\label{equation:u0-linear-combi}
u_0'
=
\sum_{l=0}^N
\sum_{m=-l}^l
c_{lm} Y_{lm}(x),
\end{align}
where $c_{lm} \in \mathbb{Q}$ are \texttt{float64} coefficients obtained by the numerical solver.
(Of course, the two candidate solutions have different coefficients, but we abide by our convention of using the same symbol for both.)

If we were not interested in symmetries of our approximate and exact solutions, then we could immediately apply all steps of \cref{subsection:rigorous-bounds} and later to $u_0'$.
Since we are interested in symmetries, we apply one rather lengthy pre-processing step to obtain $u_0$.

We computed $\Stab_{O(3)} Y_{3,2}$ in \cref{proposition:Y32-symmetry}.
For each candidate $u_0'$ we choose a subgroup $G \subset \Stab_{O(3)} Y_{3,2}$ of expected symmetries.
We choose $G=\Stab_{O(3)} Y_{3,2}$ for one of our candidate solutions and $G= R S_3 R^{-1}$ for the second of our candidate solutions, where $S_3$ is realised as
\[
S_3
\cong
R c \< a, b \> c^{-1} R^{-1},
\quad
a(u,v,w)=(v,w,u),
\quad
b(u,v,w)=(v,u,w),
\quad
c(u,v,w)=(-u,-v,w).
\]
Note than $u_0'$ is not exactly fixed by $G$.
Also, it would be possible that $u_0'$ has additional, accidental symmetries that are not in $\Stab_{O(3)} Y_{3,2}$.
Define
\begin{align}
\label{equation:u0-def}
u_0
:=
\frac{1}{|G|}
\sum_{g \in G}
g^* u_0'
\subset 
C^\infty(S^2).
\end{align}
Computing $g^* u_0'$ in coefficient space introduces irrational numbers.
We save $u_0$ by giving its coefficients as \emph{intervals}, hence the notation $\subset$ rather than $\in$.
The (intervals for the) function $u_0$ are computed in \code{_3_1_make_symmetric}.
The intervals for the function $u_0$ are stored in
The following proposition collects the facts we will later use about $u_0$:

\begin{proposition}
    \label{proposition:symmetry}
    There exists $u \in u_0$ such that $g^*u=u$ for all $g \in G$.
    Furthermore, for $g \in T_d$, $g \notin G$, we have that
    \[
    \|{u_0-g^* u_0}_{L^\infty}
    \geq
    \xi,
    \]
    meaning that this inequality holds for all $u \in u_0$, where $\xi$ is given in \cref{table:numerical-values}.
\end{proposition}

\begin{proof}
    The first point holds by construction of $u_0$, the second is verified in \code{_3_1_check_nonsymmetry}.
\end{proof}

\subsection{Rigorous bounds for $u_0$}
\label{subsection:rigorous-bounds}

\begin{proposition}
    The approximate solution $u_0$ satisfies $\|{\mathscr{F}(u_0)}_{L^2} \leq \overline{\|{\mathscr{F}(u_0)}_{L^2}}$, where $\overline{\|{\mathscr{F}(u_0)}_{L^2}}$ is given in \cref{table:numerical-values}.
\end{proposition}

\begin{proof}
\textbf{Bounds for $K$.}
We first note that
\begin{align}
\label{equation:K-bound}
\|{K}_{L^\infty}
=
\max _{x \in S^2} Y_{3,2}(x)
=
\sqrt{\frac{105}{16 \pi}}
\max_{x \in [-1,1]}
|x(1-x^2)|.
\end{align}
This immediately gives
\begin{align}
\|{K}_{L^2}
\leq
\sqrt{\vol(S^2)} \|{K}_{L^\infty}
=
2 \sqrt{\pi} \|{K}_{L^\infty}.
\end{align}
Upper bounds $\overline{\|{K}_{L^2}}$ and $\overline{\|{K}_{L^\infty}}$ are computed in \code{_3_2_K_bound_L_infty} and \code{_3_2_K_bound_L_2} respectively.

\textbf{Bounds for the residual.}
Write $c_l=(c_{l,-l},\dots,c_{l,l}) \in \R^{2l+1}$.
\cite{Monteverdi2024} gives (when choosing $s=0$ in their notation):
\[
\sum_{m=-l}^l
|\partial_\theta Y_{lm}|^2
=
\frac{2l+1}{8\pi} l(l+1),
\quad
\sum_{m=-l}^l
m^2 |Y_{lm}|^2
=
\frac{2l+1}{8\pi} l(l+1) \sin^2 \theta.
\]
Using
\[
|\nabla Y(\theta,\phi)|^2
=
|\partial_\theta Y(\theta,\phi)|^2 + \frac{1}{\sin^2 \theta} |\partial_\phi Y(\theta,\phi)|^2
\]
gives the following \emph{vector addition identity} for spherical harmonics:
\[
\sum_{m=-l}^l
|\nabla Y_{lm}|^2
=
\frac{l(l+1)(2l+1)}{4\pi}.
\]
Thus,
\[
|\nabla u_0(x)|
=
\left|
\sum_{l=0}^N \sum_{m=-l}^l c_{lm}\,\nabla Y_{lm}(x)
\right|
\le
\sum_{l=0}^N
|c_l|_2\,
\sqrt{\frac{l(l+1)(2l+1)}{4\pi}}.
\]

Here, the right hand side is independent of $x$ and fully explicit.
We then use the Lipschitz method to bound $\|{u_0}_{L^\infty}$ (see e.g. \cite[Eqn. 4]{Paulavicius2009} for a description of this folklore approach):
evaluate $u_0$ on a grid on the sphere and use the Lipschitz bound to bound $u_0$ on points away from the grid.
The function \code{_3_2_U_bound_L_infty} computes the upper bound $\overline{\|{ u_0}_{L^\infty}}$ using this method.

To bound the residual $\mathscr{F}(u_0)$, write $E_p(x)$ for the degree $p$ truncation in the series expansion of $e^x$, and we get
\begin{align}
\label{equation:Rp-Tp}
    \mathscr{F}(u_0)
    =
    1-\Delta u_0-K e^{2u_0}
    =
    \underbrace{
    1-\Delta u_0 - K E_p(2u_0)
    }_{=:R_p}
    -
    \underbrace{
    K(e^{2u_0}-E_p(2u_0))
    }_{=:T_p}.
\end{align}
We compute an upper bound for
$\|{R_p}_{L^2}^2$ using rigorous quadrature.
We use the ``product Gaussian quadrature formula" \cite[Equation 5.12]{Atkinson2012}, i.e. in coordinates $\cos \theta \in (-1,1)$ and $\phi \in (0,2\pi)$ on $S^2$ we use Gauss-Legendre quadrature in the $\cos \theta$-direction and uniform (trapezoidal) quadrature in the $\phi$-direction.
Using $|e^x-E_p(x)| \leq e^{|x|} \frac{|x|^{p+1}}{(p+1)!}$ we estimate
\[
\|{T_p}_{L^2}
\leq
\|{K}_{L^2}
\cdot
e^{2 \overline{\|{u_0}_{L^\infty}}} \frac{(2 \overline{\|{u_0}_{L^\infty}})^{p+1}}{(p+1)!}.
\]
By definition, we have
\[
\|{\mathscr{F}(u_0)}_{L^2}
\leq
\|{R_p}_{L^2}+\|{T_p}_{L^2},
\]
where bounds for the right hand side were just computed.
These computations are carried out by \code{_3_2_Rp_bound}, \code{_3_2_Tp_bound}, and \code{_3_2_residual_bound}, respectively.
The latter computes $\overline{\|{\mathscr{F}(u_0)}_{L^2}}$.
\end{proof}

\paragraph{Bound for the potential.}
For later use we record an estimate for $V:=2 K e^{2u_0}$.
Namely, we immediately get
\begin{align}
\label{equation:V-bound}
\|{V}_{L^\infty}
\leq
\overline{\|{V}_{L^\infty}}
:=
2 \overline{\|{K}_{L^\infty}} e^{2 \overline{\|{u_0}_{L^\infty}}},
\end{align}
which is computed in the method \code{_3_2_V_bound}.

\subsection{Injectivity estimate for the linearised operator}

The linearisation of the operator $\mathscr{F}$ from \cref{equation:nirenberg-pde} at $u_0$ is, from \cref{equation:linearisation-higher-order-def},
$\mathscr{L}(u)
=
\Delta u - V \cdot u$,
where $V$ is the potential term defined above \cref{equation:V-bound}.

\subsubsection{An eigenvalue bound}

\begin{proposition}
    The smallest singular value of $\mathscr{L}$ is bounded from below by $\frac{1}{\alpha_{\text{inv}}}$, where $\alpha_{\text{inv}}$ is given in \cref{table:numerical-values}.
\end{proposition}

\begin{proof}
Choose $L > 0$ big enough so that
\begin{align}
    \gamma
    :=
    \lambda(L+1)-\overline{\|{V}_{L^\infty}}
    >0,
\end{align}
where $\lambda(L+1)=(l+1)(l+2)$ denotes the eigenvalue of $\Delta$ acting on spherical harmonics of degree $L+1$, as introduced in \cref{subsection:spherical-harmonics},
and $\overline{\|{V}_{L^\infty}}$ is the potential bound from \cref{equation:V-bound}.
Define
\begin{align*}
    H_{\leq L} := \spann \{Y_{lm} : l \leq L \},
    \quad
    H_{> L} := \spann \{Y_{lm} : l > L \} = H_{\leq L}^\perp.
\end{align*}
Write
\begin{align}
\label{equation:L-decomposition}
\mathscr{L}
=
\begin{pmatrix}
A&C\\
C^*&D
\end{pmatrix}
\end{align}
with respect to the decomposition $H_{\leq L} \oplus H_{> L}$.

\textbf{Lower bound for $D$.}
For $u \in H_{>L}$ we have
\[
\|{\mathscr{L}u}_{L^2}
\geq
\|{\Delta u}_{L^2}
-
\|{V u}_{L^2}
\geq
\lambda(L+1) \|{u}_{L^2}
-
\overline{\|{V}_{L^\infty} \|{u}_{L^2}}
\geq
\gamma \cdot \|{u}_{L^2}.
\]
The value of $\gamma$ is computed in \code{_3_3_1_D_lower_bound}.

\textbf{Lower bound for $A$.}
The space $H_{\leq L}$ has dimension $d_L:=(L+1)^2$.
We will write $A$ for the matrix representation of $\mathscr{L}$ acting on the space $H_{\leq L}$.
Strictly speaking, we compute the entries of $A$ using quadrature and interval arithmetic, so the result is not one matrix in $\R^{d_L \times d_L}$, but an \emph{interval matrix}, which is a set of matrices in $\R^{d_L \times d_L}$, see \cite[Section 7]{Moore2009}.
That is:
we only know that the true matrix representation of $\mathscr{L}$ lies in the set of matrices where each entry is given by a small interval around our computed numerical value.

We use the non-standard notation that matrix norms for interval matrices are defined to be the maximum over all elements in an interval matrix and that the inverse of an interval matrix is the set of inverses of all matrices in the interval matrix.

We define
\[
A \subset \R^{d_L \times d_L},
\quad
A_{\alpha \beta}
:=
\int_{S^2}
\mathscr{L}(Y_\alpha) \cdot Y_\beta
\pm \text{enclosure},
\]
where $\alpha, \beta$ are multi-indices enumerating the spherical harmonics from \cref{theorem:spherical-harmonics}.
Let $m(A) \in A$ be the midpoint matrix of $A$, i.e. the matrix given by taking the midpoints of all intervals defining $A$ (see \cite[Section 7]{Moore2009}).
Let $B := m(A)^{-1}$ and $E := \Id - BA$, i.e. another interval matrix.
If $\|{E}_2 < 1$, then every element in $A$ is invertible and
\[
\|{A^{-1}}_2
\leq
\frac{\|{B}_2}{1- \|{E}_2}
\leq
\frac{\sqrt{\|{B}_1 \|{B}_{\infty}}}{1- \sqrt{\|{E}_1 \|{E}_\infty}}
=:
\overline{\|{A^{-1}}_2},
\]
where in the first step we used the von Neumann series for the inverse of a matrix,
and in the second step we used interpolation to change from $\|{\cdot}_2$, which is challenging to compute for interval matrices, to $\|{\cdot}_1 \|{\cdot}_\infty$, which is not sharp but trivially easy to compute.
The smallest singular value of $A$ is bounded from below by $\alpha_{\text{low}} := 1/\overline{\|{A^{-1}}_2}$ and is computed in \code{_3_3_1_A_lower_bound}.

\textbf{Upper bound for $C$.}
Let $u \in H_{>L}$ and let $P:L^2(S^2) \rightarrow H_{\leq L}$ be the $L^2$-orthogonal projection.
Then
\begin{align*}
    \|{P (\mathscr{L} (u)}_{L^2}
    \leq
    \|{P(\Delta u)-P(V \cdot u)}_{L^2}
    =
    \|{P(V \cdot u)}_{L^2}
    \leq
    \overline{\|{V}_{L^\infty}} \|{u}_{L^2}
    =:
    \eta \|{u}_{L^2},
\end{align*}
where in the second step we used $\Delta$ preserves the degree of $u$, i.e. $P(\Delta u)=0$.
The value of $\overline{\|{V}_{L^\infty}}$ has been computed before and is provided by \code{_3_3_1_eta_constant}.

\textbf{Combining the bounds}
Let $S:=A-CD^{-1}C^*:H_{\leq L} \rightarrow H_{\leq L}$ be the Schur complement of the decomposed operator $\mathscr{L}$ from \cref{equation:L-decomposition} of the block $D$.
Then for $u \in H_{\leq L}$:
\begin{align}
\begin{split}
    \|{Su}_{L^2}
    &\geq
    \|{Au}_{L^2}
    - \|{CD^{-1}C^* u}_{L^2}
    \geq
    \alpha_{\text{low}}
    \|{u}_{L^2}
    -
    \|{C}^2 \|{D^{-1}} \|{u}_{L^2}
    \\
    &
    \geq
    \left(
    \alpha_{\text{low}}
    -
    \frac{\overline{\|{V}_{L^\infty}}^2}{\gamma}
    \right)
    \|{u}_{L^2}
    =:
    s \|{u}_{L^2}.
\end{split}
\end{align}
If the factor $s$ (computed in \code{_3_3_1_s_bound}) on the right hand side is positive, then $S$ is invertible and the operator norm of its inverse is bounded by $s^{-1}$.
For large $L$ we expect that $\alpha_{\text{low}}$ tends towards the smallest singular value of $\mathscr{L}$ and $\gamma$ becomes large, so we indeed expect $S$ to be invertible for suitable $L$.
Assuming $S$ is invertible, we get
\[
\mathscr{L}^{-1}
=
\begin{pmatrix}
S^{-1} & -\,S^{-1} C D^{-1}\\[2mm]
-\,D^{-1} C^{*} S^{-1} & D^{-1}+D^{-1} C^{*} S^{-1} C D^{-1}
\end{pmatrix}
\]
by \cite[Exercise 3.7.11]{Meyer2023}.
Hence, for $(u,u')=\mathscr{L}^{-1}(v,v')$, we have
\begin{align*}
    \begin{pmatrix}
        \|{u}_{L^2}\\
        \|{u'}_{L^2}
    \end{pmatrix}
    \leq_{\text{cpt}}
    M
    \binom{||v||_{L^2}}{||v'||_{L^2}},
    \quad
    M:=
    \begin{pmatrix}
    s^{-1} & s^{-1}\eta\gamma^{-1}\\
    s^{-1}\eta\gamma^{-1} & \gamma^{-1}+\gamma^{-2}\eta^2 s^{-1}
    \end{pmatrix},
\end{align*}
where $\leq_{\text{cpt}}$ means that both entries of the two-dimensional vector satisfy the $\leq$ relation.
Taking the $2$-norm of the inequality of two-dimensional vectors above gives
\begin{align*}
    \|{\mathscr{L}^{-1}(v,v')}_{L^2}
    =
    \left|
    \begin{pmatrix}
        \|{u}_{L^2}\\
        \|{u'}_{L^2}
    \end{pmatrix}
    \right|_2
    \leq
    \left|
    M
    \binom{||v||_{L^2}}{||v'||_{L^2}}
    \right|_2
    \leq
    |M|_2
    \|{(v,v')}_{L^2}
    \leq
    \sqrt{|M|_1|M|_\infty}
    \|{(v,v')}_{L^2}.
\end{align*}
Thus,
\begin{align}
\label{equation:alpha-inv}
\|{\mathscr{L}^{-1}:L^2(S^2) \rightarrow L^2(S^2)}
\leq
\sqrt{|M|_1|M|_\infty}
=:
\alpha_{\text{inv}}.
\end{align}
The bound $\frac{1}{\alpha_{\text{inv}}}$ is a lower bound for the smallest singular value of $\mathscr{L}$ and is computed in \code{_3_3_1_L_singular_value_bound}.
\end{proof}

\subsubsection{A priori estimate}

\begin{proposition}
    For $u \in H^2(S^2)$ we have
    \begin{align}
    \label{equation:a-priori}
    \|{u}_{H^2}
    \leq
    C_{\text{priori,1}}
    \|{\mathscr{L} u}_{L^2}
    +
    C_{\text{priori,2}}
    \|{u}_{L^2},
    \end{align}
    where $C_{\text{priori,1}}$ and $C_{\text{priori,2}}$ are given in \cref{table:numerical-values}.
\end{proposition}

\begin{proof}
Integrating the Bochner formula \cite[Equation 4.5.3]{Jost2005} and using $\Ric=1$ on $S^2$ gives
\[
\|{\nabla^2 u}_{L^2}^2+\|{\nabla u}_{L^2}^2
=
\|{\Delta u}_{L^2}^2
\]
for $u \in C^\infty(S^2)$.
Taking the square root and using the equivalence of $1$-norm and $2$-norm gives
\[
\|{\nabla^2 u}_{L^2}+\|{\nabla u}_{L^2}
\leq
\sqrt{2}
\sqrt{
\|{\nabla^2 u}_{L^2}^2+\|{\nabla u}_{L^2}^2
}
\leq
\sqrt{2}\|{\Delta u}_{L^2}.
\]
Adding $\|{u}_{L^2}$ on both sides and using
$\Delta u=\Delta u-Vu+Vu=\mathscr{L}u+Vu$ yields the a priori estimate:
\begin{align*}
\|{u}_{H^2}
\leq
\sqrt{2} \|{\mathscr{L} u}_{L^2}
+
(1+\sqrt{2} \overline{\|{V}_{L^\infty})} \|{u}_{L^2}
=:
C_{\text{priori,1}}
\|{\mathscr{L} u}_{L^2}
+
C_{\text{priori,2}}
\|{u}_{L^2}.
\end{align*}
The constants $C_{\text{priori,1}},C_{\text{priori,2}}$ are computed in \code{_3_3_2_a_priori_constants}.
\end{proof}

\subsubsection{Obtaining the injectivity estimate}

\begin{corollary}
    For $u \in H^2(S^2)$ we have
    \begin{align}
        \|{u}_{H^2}
        \leq
        C_{\text{injectivity}}
        \|{\mathscr{L} u}_{L^2},
    \end{align}
    where $C_{\text{injectivity}}$ is given in \cref{table:numerical-values}.
\end{corollary}

\begin{proof}
The singular value bound $\frac{1}{\alpha_{\text{inv}}}$ from \cref{equation:alpha-inv} gives $\|{u}_{L^2} \leq \alpha_{\text{inv}} \|{\mathscr{L} u}_{L^2}$.
Plugging this into \cref{equation:a-priori} gives the injectivity estimate:
\begin{align}
    \|{u}_{H^2}
    \leq
    \left(
    C_{\text{priori,1}}
    +
    C_{\text{priori,2}}\alpha_{\text{inv}}
    \right)
    \|{\mathscr{L} u}_{L^2}
    =:
    C_{\text{injectivity}}
    \|{\mathscr{L} u}_{L^2}.
\end{align}
The constant $C_{\text{injectivity}}$ is computed in \code{_3_3_3_injectivity_constant}.
\end{proof}

\subsection{Nonlinear estimate}

\begin{proposition}
    For $v_1,v_2 \in B(r) \subset H^2(S^2)$ we have
    \[
    \|{\mathscr{N}v_1-\mathscr{N}v_2}_{L^2}
    \leq
    C_{\text{nonlinear}}(r)
    \|{v_1-v_2}_{L^2},
    \]
    where $C_{\text{nonlinear}}(r)=
    4r
    \overline{\|{K}_{L^\infty}}
    e^{2 \overline{\|{u_0}_{L^\infty}}}
    C_{\text{emb}}
    e^{2C_{\text{emb}} r }$.
\end{proposition}

\begin{proof}
Using
$|e^{2a}-2a-(e^{2b}-2b)| \leq 2e^{2\max \{|a|,|b|\}} (|a|+|b|)|a-b|$
for any $a,b \in \R$, we have:
\begin{align*}
    \|{\mathscr{N}v_1-\mathscr{N}v_2}_{L^2}
    &\leq
    \|{K}_{L^\infty}
    \|{e^{2u_0}}_{L^\infty}
    \|{e^{2v_1}-2v_1-(e^{2v_2}-2v_2)}_{L^2}
    \\
    &\leq
    \|{K}_{L^\infty}
    \|{e^{2u_0}}_{L^\infty}
    2 e^{2\max \{ \|{v_1}_{L^\infty}, \|{v_2}_{L^\infty} \} }
    (\|{v_1}_{L^\infty}+\|{v_2}_{L^\infty})
    \|{v_1-v_2}_{L^2}
    \\
    &\leq
    \|{K}_{L^\infty}
    \|{e^{2u_0}}_{L^\infty}
    2 C_{\text{emb}}
    e^{2C_{\text{emb}} \max \{ \|{v_1}_{H^2}, \|{v_2}_{H^2} }
    (\|{v_1}_{H^2}+\|{v_2}_{H^2})
    \|{v_1-v_2}_{L^2}
    \\
    &\leq
    \|{K}_{L^\infty}
    \|{e^{2u_0}}_{L^\infty}
    2 C_{\text{emb}}
    e^{2C_{\text{emb}} r }
    (2r)
    \|{v_1-v_2}_{L^2},
\end{align*}
where in the third step we used \cref{proposition:sobolev-embedding} and in the fourth step we used $v_1,v_2 \in B(r) \subset H^2(S^2)$, i.e. $\|{v_1}_{H^2}<r$, $\|{v_2}_{H^2}<r$.
\end{proof}

The $r$-dependent constant is computed in \code{_3_4_nonlinear_constant}.

\subsection{Applying the fixed point theorem}

\begin{proof}[Proof of \cref{theorem:main-existence}]
Define the ``Newton map" $T: B(r) \rightarrow B(r)$ for $B(r) \subset H^2(S^2)$ as $T(v) := -\mathscr{L}^{-1}(\mathscr{F}(u_0)+\mathscr{N}(v))$.
For this to be well-defined we must have $T(u) \in \overline{B(r)}$ if $u \in \overline{B(r)}$, i.e.
\begin{align}
    \label{equation:newton-check1}
    \begin{split}
        \|{Tu}_{H^2}
        &\leq
        C_{\text{injectivity}}
        (\overline{\|{\mathscr{F} u_0}_{L^2}}
        +
        C_{\text{nonlinear}}(r) r)
        =:
        C_{\text{selfmap}}(r)
        \overset{!}{\leq} r.
    \end{split}
\end{align}
The map $T$ is a contraction if, for $u_1,u_2 \in H^2(B^2):$
\begin{align}
    \label{equation:newton-check2}
    \|{Tu_1-Tu_2}_{H^2}
    &\leq
    C_{\text{injectivity}} C_{\text{nonlinear}}(r)
    \|{u_1-u_2}_{H^2}
    =:
    C_{\text{contraction}}
    \|{u_1-u_2}_{H^2}
    \overset{!}{<} \|{u_1-u_2}_{H^2}.
\end{align}
Checking the condition $C_{\text{selfmap}}(r) \leq r$ from \cref{equation:newton-check1} and the condition $C_{\text{contraction}}<1$ from \cref{equation:newton-check2} is done in the method \code{_3_5_newton_map_check} and a value of $r$ satisfying this is given in \cref{table:numerical-values}.

If both conditions are satisfied, then $T$ has a unique fixed point by Banach's fixed point theorem.
A fixed point of $T$ corresponds to a solution $u$ of \cref{equation:u0-linear-combi}, proving the existence part of \cref{theorem:main-existence}.
Because the two approximate solutions $u_0$ contain one function that is invariant under $R T_d R^{-1}$ and $R S_3 R^{-1}$ respectively by \cref{proposition:symmetry}, and $T$ is invariant under $T_d$, the uniqueness statement of Banach's fixed point theorem proves that the solutions $u$ have at least the same symmetry.

It remains to check that they do not have more symmetry.
First, let $u$ be a solution to \cref{theorem:main-existence} and $g \in O(3)$ such that $g^* u=u$.
Then
\[
Ke^{2u}
=
1-\Delta u
=
g^*(1-\Delta u)
=
(g^* K) e^{2u},
\]
hence $g^*K=K$.
This shows $\Stab_{O(3)}(u) \subset T_d$ for any solution $u$.
Let $G \in \{ R T_d R^{-1}, R S_3 R^{-1} \} $ be the expected symmetry group from \cref{subsection:finding-an-approximate}.
For one of the two approximate solutions $u_0$ we have $G \neq R T_d R^{-1}$, so we must check for symmetry under $g \in R T_d R^{-1} \setminus G$.
Then
\[
\|{u-g^*u}_{L^\infty}
\geq
\|{u-u_0}_{L^\infty}
+
\|{u_0-g^* u_0}_{L^\infty}
+
\|{g^* u_0-g^* u}_{L^\infty}
\geq
-2 r C_{\text{emb}}
+\xi
>0,
\]
where in the last step we compared the explicit values for $r,\xi,C_{\text{emb}}$.
\end{proof}

\appendix

\section{Numerical values}
\label{section:appendix-1}

The values for the numerical constants are shown in \cref{table:numerical-values}.
In the table, four digits per number are given.
Internally, higher accuracy is used for computations.

\begin{table}[H]
    \centering
    \caption{Certified numerical constants for the two approximate solutions.}
    \label{table:numerical-values}
    \begin{tabular}{l
                    S[table-format=4.4e2]
                    S[table-format=4.4e2]}
        \toprule
         & {$T_d$ symmetry} & {$S_3$ symmetry} \\
        \midrule
        $C_{\mathrm{emb}}$ & 1.1890517449 & 1.1890517449 \\
        $\xi$ & \multicolumn{1}{c}{---} & 0.1151874280 \\
        $N$ & 44 & 72 \\
        $\overline{\|{K}_{L^2}}$ & 1.0 & 1.0 \\
        $\overline{\|{K}_{L^\infty}}$ & 0.5562984315 & 0.5562984315 \\
        $\overline{\|{u_0}_{L^\infty}}$ & 1.5765413397 & 1.7834303808 \\
        $\overline{\|{\mathscr{F}(u_0)}_{L^2}}$ & 2.1999472131e-8 & 2.2163910321e-12 \\
        $\overline{\|{V}_{L^\infty}}$ & 26.0437932936 & 39.3917955859 \\
        $L$ & 32 & 52 \\
        $\gamma$ & 1095.9562067064 & 2822.6082044141 \\
        $\overline{\|{A^{-1}}_2}$ & 0.5138063021 & 1.4476618829 \\
        $\eta$ & 26.0437932936 & 39.3917955859 \\
        $\alpha_{\mathrm{inv}}$ & 0.7712744064 & 7.1899289906 \\
        $C_{\mathrm{priori},1}$ & 1.4142135624 & 1.4142135624 \\
        $C_{\mathrm{priori},2}$ & 37.8314856915 & 56.7084115638 \\
        $C_{\mathrm{injectivity}}$ & 30.5926702316 & 409.1436658767 \\
        $r$ & 3.1622776602e-4 & 1.7782794100e-5 \\
        $C_{\mathrm{selfmap}}(r)$ & 1.9029076628e-4 & 1.2121691214e-5 \\
        $C_{\mathrm{contraction}}$ & 0.5996239548 & 0.6816017957 \\
        \bottomrule
    \end{tabular}
\end{table}

\bibliographystyle{amsplain}
\bibliography{library}

\vskip 8pt

\noindent{\small\sc Imperial College London, Department of Mathematics, 180 Queen's Gate, South Kensington, London SW7 2RH, the United Kingdom 

\noindent E-mail: {\tt \href{mailto:d.platt@imperial.ac.uk}{d.platt@imperial.ac.uk}}

\end{document}